\documentclass[12pt, twoside]{article}
\usepackage{amsmath,amsthm,amssymb}
\usepackage{times}
\usepackage{enumerate}

\numberwithin{equation}{section}

\def\titlerunning#1{\gdef\titrun{#1}}
\makeatletter
\def\author#1{\gdef\autrun{\def\and{\unskip, }#1}\gdef\@author{#1}}
\def\address#1{{\def\and{\\\hspace*{18pt}}\renewcommand{\thefootnote}{}%
\footnote {#1}}%
\markboth{\autrun}{\titrun}}
\makeatother

\def\subjclass#1{{\renewcommand{\thefootnote}{}%
\footnote{\emph{Mathematics Subject Classification (2010):} #1}}}

\newtheoremstyle{mystyle}{}{}{\slshape}{2pt}{\scshape}{.}{ }{} 

\newtheorem{thm}{Theorem}[section]
\newtheorem{cor}[thm]{Corollary}

\newtheorem{prop}[thm]{Proposition}
\newtheorem{lemme}[thm]{Lemma}
\newtheorem{fait}[thm]{Fact}

\theoremstyle{definition}
\newtheorem{defi}[thm]{Definition}
\theoremstyle{mystyle}
\newtheorem{ex}[thm]{Exemple}

\theoremstyle{remark}

\newcommand{\pred}{\mathbf P}

\newcommand{\llg}{\langle}
\newcommand{\rrg}{\rangle}
\newcommand{\impl}{\rightarrow}
\newcommand{\monster}{\mathcal U}

\newcommand{\ignore}[1]{}

\DeclareMathOperator{\tp}{tp}

\DeclareMathOperator{\Mor}{Mor}
\DeclareMathOperator{\op}{op}

\def\indsym#1#2{%
 \setbox0=\hbox{$\m@th#1x$}%
 \kern\wd0%
 \hbox to 0pt{\hss$\m@th#1\mid$\hbox to 0pt{$\m@th#1^{#2}$\hss}\hss}%
 \lower.9\ht0\hbox to 0pt{\hss$\m@th#1\smile$\hss}%
 \kern\wd0}

\def\nindsym#1#2{%
 \setbox0=\hbox{$\m@th#1x$}%
 \kern\wd0%
 \hbox to 0pt{\hss$\m@th#1\not$\kern1.4\wd0\hss}
 \hbox to 0pt{\hss$\m@th#1\mid$\hbox to 0pt{$\m@th#1^{#2}$\hss}\hss}%
 \lower.9\ht0\hbox to 0pt{\hss$\m@th#1\smile$\hss}%
 \kern\wd0}

\begin{document}


\baselineskip=17pt


\titlerunning{Type decomposition in NIP}

\title{Type decomposition in NIP theories}

\author{Pierre Simon}

\date{}

\maketitle

\address{University of California, Berkeley\\Mathematics Department\\Evans Hall\\Berkeley, CA, 94720-3840}

\subjclass{03C45}


\begin{abstract}A first order theory is NIP if all definable families of subsets have finite VC-dimension. We provide a justification for the intuition that NIP structures should be a combination of stable and order-like components. More precisely, we prove that any type in an NIP theory can be decomposed into a stable part (a generically stable partial type) and an order-like quotient.\end{abstract}
\maketitle
\section*{Introduction}

A family $\mathcal S$ of subsets of a set $X$ is said to have finite VC-dimension if there is an integer $N$, such that for any $X_0\subseteq X$ of size $N$, the restriction of $\mathcal S$ to $X_0$ is strictly smaller than the full power set of $X_0$. The name VC-dimension comes from the seminal paper of Vapnik and Chervonenkis \cite{vc} in which they prove that families of finite VC-dimension satisfy a uniform law of large numbers. This notion was introduced independently at about the same time in model theory by Shelah \cite{Sh10} under the name NIP (Negation of the Independence Property). A first order structure $M$ is NIP if all uniformly definable families of subsets of $M$ have finite VC-dimension. Classical NIP structures include algebraically closed fields, abelian groups, real closed fields (and more generally \emph{o-minimal} structures), algebraically closed valued fields and fields $\mathbb Q_p$ of $p$-adic numbers.

A subclass of NIP structures which plays a central role in model theory is that of stable structures, example of which include abelian groups, algebraically closed fields, separably closed fields... Stable structures exhibit properties characteristic of algebraic geometry: one can define dimensions on definable sets (possibly ordinal-valued), there is a canonical notion of independence, called forking-independence and with it comes the notion of a generic point of definable sets.

Another important subclass is that of o-minimal structures: a structure is o-minimal if it is equipped with a definable linear order $<$ such that every definable subset of the line is a finite union of open intervals and points. O-minimality has proved to be a very efficient framework for \emph{tame} real geometry: the condition of o-minimality forbids topological pathologies at the definable level, such as space-filling curves or nowhere differentiable functions.

Algebraically closed valued fields (ACVF) are often presented as the prototypical NIP structures since they exhibit both the phenomena of stability (seen in the residue field) and o-minimality (the value group). In fact, one often seeks to understand NIP structures starting from the stable and o-minimal situations, which are well understood, and looking for common properties (this was suggested by Shelah, see e.g. \cite[4.1]{Sh715}). In \cite{distal}, we set out to give a precise meaning to this intuition with the vague goal of decomposing an NIP structure into a stable part and an order-like part. The first step of this program was to define a class of structures in which the \emph{stable part} is trivial, even without knowing what the \emph{stable part} would be in general. This led to the definition of \emph{distal} structures, which thus correspond to the opposite extreme to stability. Typical distal structures are o-minimal structures and the field $\mathbb Q_p$ of $p$-adic numbers. Distal structures can be thought of as order-like, or purely-unstable. From a more geometric point of view, we can think of distal structures as being related to semi-algebraic geometry the same way stable structures are related to algebraic geometry: they are meant to abstract the typical combinatorial properties of semi-algebraic structures such as $\mathbb R$ or $\mathbb Q_p$.

Distal structures are characterized by the fact that every type $p(x)=\tp(a/A)$ is \emph{compressible}: for any formula $\phi(x;y)$, there is some formula $\zeta(x;t)$ such that for any finite $A_0\subseteq A$, there is $e\in A$ with $\zeta(x;e)\in p$ and $\zeta(x;e)\vdash \tp_\phi(a/A_0)$. In other words, we can uniformly compress every finite part $\tp_\phi(a/A_0)$ of $\tp_\phi(a/A)$ into a formula $\zeta(x;e)$.

Having defined the notion of order-like, the second part of the program involves decomposing an arbitrary NIP structure. This can be tried at various levels. In the paper \cite{distal} we developed some tools to decompose types over indiscernible sequences and over saturated models. We showed in both cases that one could construct some kind of \emph{stable part} over which the type behaved like in a distal theory. In the present paper, we realize our goal by building such a decomposition for types over arbitrary sets of parameters (Theorem \ref{th_main}). The \emph{stable part} that we obtain is what we call a generically stable partial type. The statement is already interesting (and not easier to prove) if we weaken the condition of generic stability to merely asking that the partial type is Ind-definable. Here is a corollary of our main theorem that is easy to state:

\begin{thm}\label{th_mainintro}
Let $T$ be NIP and let $p(x)=\tp(a/A)$ be any type. Given a formula $\phi(x;y)$, there are formulas $\zeta(x;t)\in L$ and $\delta(x;t,y)\in L(A)$ such that:

\textbf{(1)} Definability of the $\delta$-type:  for all $(e,b)\in A^{|t|+|y|}$, $\delta(a;e,b)$ holds.

\textbf{(2)} Relative compressibility: For every finite $A_0\subseteq A$, there is $e\in A$ such that $\zeta(x;e)\in p$ and for all $b\in A_0^{|y|}$, either $\zeta(x;e)\wedge \delta(x;e,b)\vdash \phi(x;b)$ or $\zeta(x;e)\wedge \delta(x;e,b)\vdash \neg \phi(x;b)$.


\textbf{(3)} Uniformity: If we write $\delta(x;y,t)=\delta_0(x,y,t;d)$ with $d\in A$ and $\delta_0\in L$, then $\delta_0$ and $\zeta$ depend only on $\phi$ and neither on $A$ nor $a$.
\end{thm}

As a consequence, we obtain a more explicit construction of honest definitions and also prove the existence of non-realized compressible types in any unstable NIP theory.

\smallskip
Our program of decomposing types was strongly influenced by various works of Shelah. The idea that types in NIP can be decomposed into a stable-like part and an order-like one appears in \cite{Sh900} and \cite{Sh950}, where this intuition is explicitely stated and important results supporting it are proved. Most notably in \cite{Sh950}, Shelah proves a decomposition theorem for types over a saturated model and deduces from it that under the NIP assumption, the number of types up to automorphisms is small (and this characterizes NIP). Although our work was inspired by that of Shelah, our approach is quite different and our theorem neither implies nor is implied by those of Shelah. The two decomposition theorems can be seen as complementing each other. Whereas Shelah \cite{Sh950} considers types over saturated models and studies them up to automorphisms, we consider arbitrary types, and describe them up to elementary equivalence. In Shelah's decomposition, the stable-like part is a type finitely satisfiable over a small set with no additional stable-like properties. In fact, assuming distality does not seem to help in simplifying his proof. Our stable-like part is a generically stable partial type, which is a stronger condition. In particular, it is an object invariant over a set of size $|T|$. A fair share of the hard work in \cite{Sh950} has to do with understanding what happens when a type is orthogonal to all types finitely satisfiable over very small sets (of size $\leq \beth_\omega$), but not orthogonal to some type over a small set (of size less than that of the saturated model). This intermediate scale disappears when we take a saturated elementary extension of the type and therefore is not involved in our work. From the point of view of Shelah's decomposition, our analysis can be thought of as looking more closely at what happens at the very small scale.

It is tempting to think that the two results could be combined (for types over saturated models), but we have not found any way of doing so.

\smallskip
The paper is organized as follows: In the first section, we set our notations and recall some properties of indiscernible sequences in NIP theories, in particular the theory of domination from \cite{distal}. Section \ref{sec_genstable} introduces generically stable partial types. Most of it makes no assumption of NIP. In Section \ref{sec_compr}, we define compressible types and prove basic statements about them. Finally Section \ref{sec_main} states and proves the decomposition theorem along with a few corollaries.

\section{Preliminaries}\label{sec_prel}

Throughout this paper, $T$ is a complete first order theory in a language $L$. We let $\monster$ be a monster model, which is $\bar \kappa$-saturated and $\bar \kappa$-strongly homogeneous for some large enough $\bar \kappa$. All sets of parameters considered have size smaller that $\bar \kappa$. If $A\subset \monster$ and $\phi(x)\in L(\monster)$ is a formula, by $\phi(A)$, we mean the set of tuples $a\in A^{|x|}$ satisfying $\phi(x)$.

We use the notation $\phi^0$ to mean $\neg \phi$ and $\phi^1$ to mean $\phi$. If $\phi(x;y)\in L$, $\tp_{\phi}(a/A)$ is the set of instances of $\phi(x;y)$ and $\neg \phi(x;y)$ in $\tp(a/A)$.

By an $A$-invariant type, we mean a global type $p$ which is invariant under automorphisms fixing $A$ pointwise. If $p(x)$ and $q(y)$ are both $A$-invariant, we can define the type $p(x)\otimes q(y)$ whose restriction to any set $B\supseteq A$ is $\tp(a,b/B)$, where $b\models q|B$ and $a\models p|B b$. It is also an $A$-invariant type. A Morley sequence of $p$ over $A$ is a sequence $I=(a_i:i\in \mathcal I)$ such that for each $i\in \mathcal I$, $a_i\models p|Aa_{<i}$. A Morley sequence of $p$ over $A$ is indiscernible over $A$ and all Morley sequences of $p$ over $A$ indexed by the same order have the same type over $A$.

Finally, if $p$ is an $A$-invariant type and $q$ is any type over a base $B\supseteq A$, then we define $p(x)\otimes q(y)\in S_{xy}(B)$ as $\tp(a,b/B)$, where $b\models q$ and $a\models p|Bb$.

\subsection{Indiscernible sequences}
We set here some terminology concerning indiscernible sequences.

If $I$ is an indiscernible sequence, we let $\op(I)$ denote the sequence $I$ indexed in the opposite order. If $I$ is an endless indiscernible sequence and $T$ is NIP, let $\lim(I)$ denote the limit type of $I$: the global $I$-invariant type defined by $\phi(x)\in \lim(I)$ if $\phi(I)$ is cofinal in $I$. Observe that if $I_1$ is a Morley sequence of $\lim(I)$ over $I$, then $I+op(I_1)$ is indiscernible.

A cut $\mathfrak c =(I_0,I_1)$ of $I$ is a pair of subsequences of $I$ such that $I_0$ is an initial segment of $I$ and $I_1$ the complementary final segment, {\it i.e.}, $I=I_0+I_1$. If $J$ is a sequence such that $I_0+J+I_1$ is indiscernible (over $A$), we say that $J$ \emph{fills} the cut $\mathfrak c$ (over $A$). To such a cut, we can associate two limit types: $\lim(I_0)$ and $\lim(\op(I_1))$ (which are defined respectively if $I_0$ and $\op(I_1)$ have no last element). The cut $(I_0,I_1)$ is \emph{Dedekind} if both $I_0$ and $\op(I_1)$ have infinite cofinalities, in particular are not empty.


We now recall the important theorem about shrinking of indiscernibles and introduce a notation related to it (see e.g. \cite[Chapter 2]{NIPBook}).

\begin{defi}
A finite convex equivalence relation on $\mathcal I$ is an equivalence relation $\sim$ on $\mathcal I$ which has finitely many classes, all of which are convex subsets of $\mathcal I$.
\end{defi}

\begin{prop}[Shrinking of indiscernibles]\label{shrinking1}
Let $( a_t)_{t\in \mathcal I}$ be an indiscernible sequence. Let $d$ be any tuple and $\phi(y_0,..,y_{n-1};d)$ a formula. There is a finite convex equivalence relation $\sim_\phi$ on $\mathcal I$ such that given: 

-- $t_0<\ldots<t_{n-1}$ in $\mathcal I$;

-- $s_0<\ldots<s_{n-1}$ in $\mathcal I$ with $t_k \sim_\phi s_k$ for all $k$;

\noindent
we have $\phi(a_{t_0},..,a_{t_{n-1}};d) \leftrightarrow \phi(a_{s_0},...,a_{s_{n-1}};d)$.

Furthermore, there is a coarsest such equivalence relation.
\end{prop}

Given $I=( a_t)_{t\in \mathcal I}$, $\phi(y_0,\ldots,y_{n-1};d)$ as above, we let $\textsf T(I,\phi)$ denote the number of equivalence classes in the coarsest $\sim_\phi$ given by the proposition. By compactness, the number $\textsf T(I,\phi)$ is bounded by an integer depending only on $\phi(y_0,\ldots,y_{n-1};z)$. (More precisely, fix some countable dense order $\mathcal I$. Then by the proposition and compactness, there is a bound on $\textsf T(I,\phi)$ for sequences indexed by $\mathcal I$. Then any sequence $I$ contains a countable subsequence with same $\textsf T(I,\phi)$, which can then be extended to one indexed by $\mathcal I$. This shows that the bound obtained actually applies to all sequences.)

If $I\subseteq J$ are indiscernible sequences and $A$ is any set of parameters, we write $I \unlhd_A J$ if for every $\phi(y_0,\ldots,y_{n-1};d)\in L(A)$, we have $\textsf T(I,\phi)=\textsf T(J,\phi)$. Intuitively, formulas with parameters in $A$ do not alternate more on $J$ than they do on $I$.

Note the following special cases:
\begin{itemize}
\item If $I$ is indiscernible over $A$, then $I\unlhd_A J$ simply means that $J$ is $A$-indiscernible and contains $I$.

\item If $I$ is without endpoints, $I \unlhd_A I_0+I+I_1$ is equivalent to the statement that $I_0$ is a Morley sequence in $\lim(\op(I))$ over $IA$ and $\op(I_1)$ is a Morley sequence in $\lim(I)$ over $AII_0$.

\item If $I$ is a Morley sequence of an $A$-invariant type $q$ over $A$ and $\bar b$ is a Morley sequence of $q$ over $AI$, then $I \unlhd_{A\bar b} J$ holds if and only if $J$ is a Morley sequence of $q$ over $A$ containing $I$ and $\bar b$ is a Morley sequence of $q$ over $AJ$. (This is merely a special case of the first point.)

\item Assume that $\mathfrak c$ and $\mathfrak d$ are two distinct cuts in an $A$-indiscernible sequence $I$. Let $\bar a_*$ and $\bar d$ fill $\mathfrak c$ and $\mathfrak d$ respectively, over $A$. Let $J$ be the sequence $I$ with $\bar d$ added in $\mathfrak d$. Then $I\unlhd_{A\bar a_*} J$ means that when we add both $\bar a_*$ and $\bar d$ to $I$ in their respective cuts, the resulting sequence is indiscernible.
\end{itemize}

Notice also that if $I=(a_i:i \in \mathcal I)$ is indiscernible, where the indexing order $\mathcal I$ is dense, then given any $\mathcal I \subseteq \mathcal J$, we can find $J=(a_i:i\in \mathcal J)$ extending $I$ such that $I\unlhd_A J$. This can be seen by a simple compactness argument. We can also build $J$ more explicitly as follows: let $M$ be a model containing $AI$. Consider the pair $(M,I)$ where $I$ is named by a predicate. Take a sufficiently saturated elementary extension $(M,I)\prec (M',I')$. Then $I'$ is $A$-indiscernible and $I\unlhd_A I'$. By saturation, one can find a subsequence $J$ of $I'$ which has the required order type.

\subsection{Domination in indiscernible sequences}

In the course of the proof of the decomposition theorem, we will need the theory of domination in indiscernible sequences presented in \cite{distal}. We recall it here.

\begin{defi}[Domination]\label{def_stablebase}
Let $q$ be an $A$-invariant type and let $I$ be a dense indiscernible Morley sequence of $q$ over $A$. Let $\bar b$ be a Morley sequence of $q$ over $AI$ and $\mathfrak c$ a Dedekind cut of $I$ filled by a dense sequence $\bar a_* = \llg a_t:t\in \mathcal I\rrg$. We say that $\bar a_*$ dominates $\bar b$ over $(I,A)$ if: For every Dedekind cut $\mathfrak d$ of $I$ distinct from $\mathfrak c$, and $\bar d$ a dense sequence filling $\mathfrak d$, we have, where $J$ is the sequence $I$ with $\bar d$ added in the cut $\mathfrak d$: \[I \unlhd_{A \bar a_*} J \Longrightarrow I \unlhd_{A\bar b} J.\]

We say that $\bar a_*$ \emph{strongly dominates} $\bar b$ over $(I,A)$ if for every dense extension $I'\supseteq I$ such that both $I\unlhd_{A\bar a_*} I'$  and $I\unlhd_{A\bar b} I'$ hold, and $\bar a_*$ fills a Dedekind cut of $I'$, then $\bar a_*$ dominates $\bar b$ over $(I',A)$. 
\end{defi}

Existence of strongly dominating sequences was proved in \cite[Proposition 3.6]{distal}. We give here a statement phrased slightly differently to fit our needs.

\begin{prop}\label{prop_existssbase}
Let $q$ be $A$-invariant and let $I_0+\bar a+I_1$ be a dense Morley sequence of $q$ over $A$. Let $\bar b$ be a Morley sequence of $q$ over $AI_0I_1$. Assume that $I_0$ and $I_1$ have no endpoints. Then there is some $I_0+\bar a+I_1 \unlhd_{A} J_0 + \bar a_0 +\bar a +\bar a_1 +J_1$, $\bar a_0$ contains $I_0$ and $\bar a_1$ contains $I_1$, such that $\bar b$ is a Morley sequence of $q$ over $AJ_0J_1$ and $\bar a_* := \bar a_0+\bar a +\bar a_1$ strongly dominates $\bar b$ over $(J_0+J_1,A)$.
\end{prop}
\begin{proof}
The proof is essentially the same as that of \cite[Proposition 3.6]{distal}.

Let $I_0 + \bar a + I_1$ be a dense Morley sequence of $q$ over $A$ and $\bar b$ a Morley sequence of $q$ over $AI_0I_1$. Assume that $\bar a$ does not strongly dominate $\bar b$ over $(I_0+I_1, A)$. Then there is some Morley sequence $I'$ of $q$ over $A$ containing $I_0+I_1$ such that $I_0+I_1\unlhd_{A\bar a} I'$ and $\bar b\models q|AI'$, some tuple $\bar d$ filling a Dedekind cut of $I'$ over $A$ such that $I' \unlhd_{A\bar a} I'\cup \bar d$ (where $\bar d$ is placed in its cut), but $I' \ntrianglelefteq_{A\bar b} I'\cup \bar d$. In this case, this just means that $I'\cup \bar d$ is not indiscernible over $A\bar b$. Hence there is some formula $\phi \in L(A\bar b)$ such that $\textsf T(I'\cup \bar d,\phi)>1$ and one of the classes of the corresponding convex equivalence relation $\sim_\phi$ lies entirely in the cut determined by $\bar d$ (and in fact, we must have $\textsf T(I'\cup \bar d,\phi)\geq 3$). Since $\bar d$ is placed in a cut distinct from that of $\bar a$, we have $I'\cup \bar a\ntrianglelefteq_{A\bar b} I'\cup \bar a \cup \bar d$, witness by the same new class of $\sim_\phi$.

Let $\bar a_1$ be equal to $I' \cup \bar a \cup \bar d$ and build $ \bar a_1 \unlhd_{A\bar b} I_0^1+\bar a_1+I_1^1$. Then $\bar b$ is a Morley sequence of $q$ over $AI_0^1I_1^1$.

Now iterate this construction building an increasing continuous sequence $(\bar a_i :i <\kappa)$. At every successor stage, for some formula $\phi\in L(A\bar b)$ the number $\textsf T(\bar a_i,\phi)$ increases. Since those numbers must remain finite, this process stops after less than $(|A|+|T|)^+$ stages. At the end, we have what we were looking for.
\end{proof}

The condition $\bar a_*$ strongly dominates $\bar b$ over $(I,A)$ is defined looking only at extensions of $I$. It turns out that it implies a much stronger domination which allows for arbitrary parameters. The following is a reformulation of \cite[Proposition 3.7]{distal} ($J_1$ and $J_3$ there are taken to be empty, $J_2$ there is $J_0+J_1$ here and $J_4$ there is $J_2$ here). It is stated in \cite{distal} in the case where $\bar b$ is a unique realization of $q$, but the proof goes through unchanged if $\bar b$ is a Morley sequence of $q$. We also state the hypothesis slightly differently: note that our hypothesis imply that $J_0+\bar a_*+J_1$ is indiscernible over $AdJ_2$ (since it is indiscernible over $Ad$ and $\tp(J_2/Ad+J_0+\bar a_*+J_1)$ is invariant over $Ad$). Thus the hypothesis in \cite{distal} are implied by those here.

\begin{fait}\label{prop_extstablebase}
Let $I$ be a dense Morley sequence of $q$ over $A$, $\bar b$ a Morley sequence of $q$ over $AI$. Fix some Dedekind cut $\mathfrak c$ of $I$ and $\bar a_*$ which fills $\mathfrak c$ over $A$. Assume that $\bar a_*$ strongly dominates $\bar b$ over $(I,A)$, then for any $d\in \monster$ if
\begin{itemize}
\item there is a partition $I=J_0+J_1+J_2$, all sequences infinite without endpoints, such that $J_0+\bar a_*+J_1$ is indiscernible over $Ad$ and $J_2$ is a Morley sequence of $q$ over $Ad + J_0+\bar a_*+J_1$,
\end{itemize}
then $\bar b$ is a Morley sequence of $q$ over $AId$.
\end{fait}

\section{Generically stable partial types}\label{sec_genstable}

A partial type $\pi(x)$ is a consistent set of formulas closed under finite conjunctions and logical consequences. We always think of $\pi$ as being over $\monster$. Given a set $A$ of parameters, $\pi|_A$ or $\pi | A$ denotes the subset of $\pi$ composed of formulas with parameters in $A$. Note that $a\models \pi|_A$ if and only if there is a global extension of $\tp(a/A)$ which satisfies $\pi(x)$.

A partial type $\pi$ is $A$-invariant if it is invariant under automorphisms of $\monster$ fixing $A$ pointwise.

\subsection{Ind-definable types}

\begin{defi}
We say that a partial type $\pi$ over $\monster$ is Ind-definable over $A$ if for every $\phi(x;y)$, the set $\{b : \phi(x;b)\in \pi\}$ is Ind-definable over $A$ ({\it i.e.}, is a union of $A$-definable sets).
\end{defi}

One can represent an $A$-Ind-definable partial type as a collection of pairs $$(\phi_i(x;y),d\phi_i(y)),$$ where $\phi_i(x;y)\in L$, $d\phi_i(y)\in L(A)$ such that $\pi(x)$ is equal to $\bigcup_i \{\phi_i(x;b) : b\in d\phi_i(\monster)\}$. The same formula $\phi(x;y)$ can appear infinitely often as $\phi_i(x;y)$.

Conversely, given a family of pairs $(\phi_i(x;y),d\phi_i(y))$, if the partial type $\pi(x)$ generated by $\bigcup_i \{\phi_i(x;b) : b\in d\phi_i(\monster)\}$ is consistent, then it is Ind-definable. Indeed if say $\psi(x;b)\in \pi(x)$, then there are $i_1,\ldots,i_n$ and $b_1,\ldots, b_n\in \monster$ such that $d\phi_{i_k}(b_k)$ holds for all $k$ and $\monster \models (\forall x)(\bigwedge \phi_{i_k}(x;b_k) \rightarrow \psi(x;b))$. Consider the formula $d\psi(y) := (\exists z_1,\ldots,z_k) \bigwedge d\phi_{i_k}(z_k) \wedge (\forall x)(\bigwedge \phi_{i_k}(x;z_k) \rightarrow \psi(x;y))$. Then $d\psi$ is over $A$ and for all $b'\in d\psi(\monster)$, we have $\psi(x;b')\in \pi$.

We say that $\pi$ is \emph{finitely definable} if there is a finite set of pairs $(\phi_i(x;y_i),d\phi_i(y_i))$ which generate $\pi(x)$ as above. We will use the notation $(\phi(x;y),d\phi(y))$ to denote the finitely definable partial type generated by $\{\phi(x;b) : b\in d\phi(\monster)\}$. Observe that the partial types $(\phi(x;y),d\phi(y))$ and $(d\phi(y)\impl \phi(x;y) ; y=y)$ are the same.

\begin{lemme}\label{lem_def}
Let $\pi(x)$ be a partial $A$-invariant type. Then $\pi$ is Ind-definable over $A$ if and only if the set $X=\{(a,\bar b) : \bar b\in \monster^{\omega}, a\models \pi|A\bar b\}$ is type-definable over $A$.
\end{lemme}
\begin{proof}
If $\pi$ is Ind-definable, then the set $X$ is type-defined by the conjunction of $d\phi(\bar y)\impl \phi(x,\bar y)$ where $(\phi,d\phi)$ ranges over all pairs of formulas in $L(A)$ such that $\phi(x,\bar b)\in \pi(x)$ for all $\bar b\models d\phi(\bar y)$.

Conversely, assume that $X$ is type-definable over $A$ and take some $\phi(x,\bar y)\in L(A)$. The set $X \cap \neg \phi(x,\bar y)$ is closed and so is its projection $Y$ to the variables $\bar y$. If $\bar b\notin Y$, then there is no $a\models \pi|A\bar b$ such that $\neg \phi(a,\bar b)$ holds. In other words, $\phi(x,\bar b)\in \pi$. And conversely, if $\phi(x,\bar b)\in \pi$, then $\bar b\notin Y$. Hence the set $\{\bar b : \phi(x,\bar b)\in \pi\}$ is open over $A$ as required.
\end{proof}

Let $\pi(x)$ and $\eta(y)$ be two $A$-invariant partial types, where $\pi$ is Ind-definable over $A$. Then there is an $A$-invariant partial type $(\pi \otimes \eta)(x,y)$ such that $(a,b)\models \pi \otimes \eta$ if and only if $b\models \eta$ and $a\models \pi |\monster b$. Indeed $(\pi \otimes \eta)(x,y)$ is generated by $\eta(y)$ along with pairs $(d\phi(y,z)\impl \phi(x;y,z), z=z)$, where the partial type $(\phi(x;y,z),d\phi(y,z))$ is in $\pi(x)$. If in addition $\eta$ is Ind-definable over $A$, then so is $\pi \otimes \eta$. As usual, we define inductively $\pi^{(n)}(x_1,\ldots,x_n)$ to be $\pi(x_n)\otimes \pi^{(n-1)}(x_1,\ldots,x_{n-1})$. All those types are Ind-definable over $A$.

\smallskip

Instead of a partial type $\pi$, one could also consider the dual ideal $I_\pi$ of $\pi$ defined as the ideal of formulas $\phi(x)$ such that $\neg \phi(x)\in \pi$. Then an $I_\pi$-wide type (namely a type not containing a formula in $I_\pi$) is precisely a type over some $A$ containing $\pi|A$. This is more consistent with the usage in model theory, but we find that it is easier to think of the partial type rather than the ideal due to the similarity between partial generically stable types to be defined soon and complete generically stable types. The reader might nonetheless find this point of view useful.\footnote{Thanks to Udi Hrushovski for pointing this out to me.}

\subsection{Generic stability}

\begin{defi}
We say that a partial type $\pi(x)$ over $\monster$ is finitely satisfiable in $A$ if any formula in it has a realization in $A$ (recall that we assume $\pi$ to be closed under conjunctions). 
\end{defi}

\begin{lemme}\label{lem_deffs}
Let $\pi$ be a partial type Ind-definable over $A$. Let $a\models \pi|A$ and $b$ such that $\tp(b/Aa)$ is finitely satisfiable in $A$. Then $a\models \pi|Ab$.
\end{lemme}
\begin{proof}
Assume not, then there is $\phi(x;y)\in L(A)$ such that $\phi(x;b)\in \pi$, but $a\models \neg \phi(x;b)$. By Ind-definability of $\pi$, there is some $\theta(y)\in L(A)$ such that $\phi(x;b')\in \pi$ for all $b'\models \theta(y)$. As $\tp(b/Aa)$ is finitely satisfiable in $A$, there is $b_0\in A$ such that $b_0 \models \neg \phi(a;y) \wedge \theta(y)$. But this contradicts the fact that $a\models \pi|A$.
\end{proof}

\begin{defi}
Let $\pi(x)$ be a partial type. We say that $\pi$ is generically stable over $A$ if $\pi$ is Ind-definable over $A$ and the following holds:

(GS) if $(a_k:k<\omega)$ is such that $a_k \models \pi|Aa_{<k}$ and $\phi(x;b)\in \pi$, then for all but finitely many values of $k$, we have $\models \phi(a_k;b)$.
\end{defi}

This definition generalizes the one for complete types; see \cite[Section 2.2.2]{NIPBook}.

\begin{prop}
Let $\pi$ be a partial type generically stable over $A$. Then:

(FS) $\pi$ is finitely satisfiable in every model containing $A$;

(NF) let $\phi(x;b)\in \pi$ and take $a\models \pi|A$ such that $\models \neg \phi(a;b)$. Then both $\tp(b/Aa)$ and $\tp(a/Ab)$ fork over $A$.
\end{prop}
\begin{proof}
(FS): Fix a model $M\supseteq A$ and some $\phi(x;b)\in \pi$. Let $(a_k:k<\omega)$ be such that $a_k\models \pi|Aa_{<k}$ for all $k$ and $\tp((a_k)/Mb)$ is finitely satisfiable in $M$. Then by (GS), for some $k$, $a_k\models \phi(x;b)$. As $\tp(a_k/Mb)$ is finitely satisfiable in $M$, there is $a\in M$, $a\models \phi(x;b)$ as required.

(NF): We first show that $\tp(b/Aa)$ divides over $A$. Let $\pi' = \pi \cup \tp(a/A)$. It is a consistent type since $a\models \pi|A$ (in fact generically stable). Let $\bar a=(a_k)_{k<\omega} \models \pi'^{(\omega)}(\bar x)$. By Ramsey and compactness, we can assume that the sequence $\bar a$ is indiscernible over $A$. Then (GS) implies that the set $\{\neg \phi(a_k;y)\wedge d\phi(y) : k<\omega\}$ is inconsistent. Hence $\tp(b/Aa)$ divides over $A$.

Now assume that $\tp(a/Ab)$ does not fork over $A$. Build $(a_k:k<\omega)$ an indiscernible sequence of realizations of $\tp(a/Ab)$ such that $\tp(a_k/Aba_{<k})$ does does fork over $A$ (we can build such a sequence by building a very long one which satisfies only the non-forking condition and then obtain an indiscernible sequence from it using Erd\H os-Rado). By transitivity of non-forking, for every $k$, $\tp(a_{>k}/Aa_k)$ does not fork over $A$. Therefore, by the previous paragraph, $a_k\models \pi|Aa_{>k}$. By (GS) this implies that for every $\phi(x;b)\in \pi$, the set $\{k : \models \neg \phi(a_k;b)\}$ is finite. As $\tp(a_k/Ab)$ is equal to $\tp(a/Ab)$ for all $k$, we obtain a contradiction.
\end{proof}

\begin{prop}\label{prop_gssym}($T$ is NIP.)
Let $\pi$ be a partial type over $\monster$ which is Ind-definable over $A$. Then $\pi$ is generically stable if and only if the following holds:

(Sym) whenever $(a_k:k<\omega)$ is indiscernible over $A$ such that $a_k\models \pi|Aa_{<k}$, then $a_k\models \pi|Aa_{\neq k}$.
\end{prop}
\begin{proof}
It is clear that (GS) implies (Sym). We show the converse.

Let $(a_k:k<\omega)\models \pi^{(\omega)}(\bar x)$ and assume that for some $\phi(x;b)\in \pi$, the set $\{ k: \models \neg \phi(a_k;b)\}$ is infinite. Without loss, $\neg \phi(a_k;b)$ holds for all $k$ and then by Ramsey and compactness we may assume that the sequence $(a_k:k<\omega)$ is indiscernible. Then by (Sym), we have $a_k \models \pi|Aa_{\neq k}$ for all $k$. We will show the following statement by induction on $l$:

For every $s \in 2^l$, there is $b_s$, $\tp(b_s/A)=\tp(b/A)$ and for $k<l$, we have $\models \phi(a_k;b_s)\iff s(k)=1$. This will contradict NIP.

For $l=1$, we set $b_{\langle 0\rangle} = b$ and as $\phi(x;b)\in \pi$, there is $b_{\langle 1\rangle}$ such that $\phi(a_0;b_{\langle 1\rangle})$ holds and $\tp(b_{\langle 1\rangle}/A)=\tp(b/A)$.

Assume we know it for $l$ and let $s \in  2^{l+1}$.
If $s(k)=0$ for all $k$, we may take $b_s = b$. Otherwise, take some $k_*<l$ such that $s(k_*)=1$. By induction hypothesis, there is $b'$, $\tp(b'/A)=\tp(b/A)$ such that for $k\leq l$, $k\neq k_*$, we have $\models \phi(a_k;b')^{s(k)}$. We have $\phi(x;b')\in \pi$ by invariance. Therefore for any formula $\theta(y)\in \tp(b/A)$, we have $(\exists y) \theta(y) \wedge \bigwedge_{k\leq l, k\neq k_*} \phi(a_k;y)^{s(k)} \wedge \phi(x;y) \in \pi(x)$. As $a_{k_*} \models \pi|Aa_{\neq k_*}$, we can find $b_s$ as required.
\end{proof}

\begin{prop}\label{prop_cargenstable}
Let $\pi$ be $A$-invariant. Assume that for all $B\supseteq A$, and for all $p\in S_x(B)$ extending $\pi|_B$, $\pi$ is included in every global non-forking extension of $p$. Then $\pi$ satisfies (GS).
\end{prop}
\begin{proof}
Assume that $\pi$ satisfies the assumption and let $(a_i:i<\omega)$ be such that $a_i \models \pi|Aa_{<i}$. If for some $\phi(x;b)\in \pi$, $\{i: \models \neg \phi(a_i;b)\}$ is infinite, then the set $\pi(x)|Aa_{<\omega} \cup \{x\neq a_i:i<\omega\} \cup \{\neg \phi(x;b)\}$ is finitely satisfiable in $\{a_i:i<\omega\}$. As such, it has a global extension $q$ finitely satisfiable in that same set. Then $q$ is a fortiori finitely satisfiable in $B=A\cup \{a_i:i<\omega\}$ and extends $\pi|_B$. But $\pi \nsubseteq q$; contradiction.
\end{proof}

\begin{lemme}\label{lem_reducepi}
Let $\pi(x)$ be generically stable over $A$ and let $\pi_0(x)\subseteq \pi(x)$ be a partial Ind-definable type, definable over some $A_0\subseteq A$. Then there is $\pi_*(x)\subseteq \pi(x)$ containing $\pi_0(x)$ which is generically stable and defined over some $A_*\subseteq A$ of size $\leq |A_0|+|T|$.
\end{lemme}
\begin{proof}
This is a simple compactness argument. As $\pi(x)$ is generically stable, for any $(\phi(x;y),d\phi(y))$ in $\pi_0(x)$, there is some $\psi(x_1\ldots, x_n )$ in $\pi^{(n)}(x_1,\ldots,x_n)$ such that $\models (\forall x_1\ldots x_n,y) (\psi(x_1\ldots x_n) \wedge d\phi(y))\impl \bigvee_{k\leq n} \phi(x_k;y)$. The formula $\psi(x_1,\ldots,x_n)$ is already in $\pi_1^ {(n)}$ for some finitely definable $\pi_1(x)\subseteq \pi(x)$. There are at most $|A_0|+|T|$ many schemes $(\phi(x;y),d\phi(y))$ in $\pi_0(x)$. Doing the same procedure for each of them, we obtain some $\pi_1(x)\subseteq \pi(x)$ Ind-definable over a set of size $|T|+|A_0|$ such that for any formula $\phi(x;d)\in \pi_0$, $\pi_1^{(\omega)}(x_0,\ldots) \wedge \bigwedge_{k<\omega} \neg \phi(x_k;d)$ is inconsistent. We may assume that $\pi_1$ contains $\pi_0$. Now iterate this construction to obtain $\pi_2(x), \pi_3(x),\ldots$. Finally set $\pi_\omega = \bigcup_{k<\omega} \pi_k$. Then $\pi_\omega$ is generically stable and contains $\pi_0$.
\end{proof}

Observe that if $\{\pi_i(x):i<\alpha\}$ is any small set of partial types, each of which is generically stable over $A$, then if $\bigcup_{i<\alpha} \pi_i(x)$ is consistent, then it is generically stable over $A$. This follows at once from the definition. If we assume that $\bigcup_{i<\alpha} \pi_i(x)|_A$ is consistent and does not fork over $A$, then we can conclude that $\bigcup_{i<\alpha} \pi_i(x)$ is consistent. This is because any global non-forking extension of $\bigcup_{i<\alpha} \pi_i(x)|_A$ will satisfy all the types $\pi_i(x)$ by (NF). In particular, if $A$ is an extension base (no type over $A$ forks over $A$), for any type $p(x)\in S(A)$ the union of all the partial types $\pi(x)$ generically stable over $A$ such that $\pi|_A \subseteq p$ is again a (consistent) generically stable partial type. As $x=x$ is such a type, there is a maximal $\pi(x)$ generically stable over $A$ such that $\pi|_A \subseteq p$.

Over arbitrary sets $A$, the situation is less clear. We can however state the following two lemmas (which will not be used later).
\begin{lemme}\label{lem_morleyseq}($T$ is NIP.)
Let $\pi(x)$, $\eta(x)$ be Ind-definable over $A$. Assume that $\pi^{(\omega)}(\bar x)|_A \cup \eta^{(\omega)}(\bar x)|_A$ is consistent, then $\pi(x)\cup \eta(x)$ is consistent.
\end{lemme}
\begin{proof}
Assume that for some $b$, $\phi(x;b)\in \pi(x)$, while $\neg \phi(x;b)\in \eta(x)$. Let $(a_i:i<\omega)\models (\pi^{(\omega)}|_A \cup \eta^{(\omega)}|_A)$. We will build inductively tuples $b_s$, $s \in 2^{<\omega}$, such that $b_s\equiv_A b$ and $\models \phi(a_i;b_s) \iff s(i)=1$ for $i$ in the domain of $s$.

As $\phi(x;b)\in \pi$, for any formula $\theta(y)\in \tp(b/A)$, we have $a_0\models (\exists y)\theta(y)\wedge \phi(x;y)$. Also as $\neg \phi(x;b)\in \eta$, $a_0 \models (\exists y)\theta(y)\wedge \neg \phi(x;y)$. Hence we can find $b_{\langle 0\rangle}$ and $b_{\langle 1\rangle}$ as required. Assume we have $b_s$ for $s\in 2^{< n}$. Let $s\in 2^{n}$. Since $ \phi(x;b_s)\in \pi(x)$, for every $\theta(y)\in \tp(b/A)$, we have $(\exists y)\theta(y)\wedge \bigwedge_{k<n} \phi(a_k;y)^{s(k)} \wedge \phi(x;y) \in \pi(x)$. As $a_n \models \pi|Aa_{<n}$, we can find $b_{s\hat{~}1}$ as required. Similarly using $\eta$ instead of $\pi$, we find $b_{s\hat{~}0}$. At the end, we contradict NIP.
\end{proof}

%

\begin{lemme}
Let $\pi(x)$ and $\eta(x)$ be generically stable over $A$. Assume furthermore that $\pi^{(\omega)}(\bar x)$ is generically stable over $A$ and that $ \pi(x)|_A \cup \eta(x) |_A$ is consistent. Then $\pi(x)\cup \eta(x)$ is consistent.
\end{lemme}
\begin{proof}
First note that the hypothesis that $\pi(x)|_A \cup \eta(x)|_A$ is consistent implies that $\pi(x)|_A \cup \eta(x)|_{AB}$ is consistent for any $B$. Indeed, if it was not consistent, there would be some $\phi(x;a)\in \pi(x)|_A$ and $\psi(x;b)\in \eta(x)|_{AB}$ whose conjuction is inconsistent. Then by compactness, there is some formula $\theta(y;a')\in \tp(b/A)$ such that $\phi(x;a)\vdash \neg (\exists y)(\theta(y;a')\wedge \psi(x;y))$. But the formula $(\exists y)(\theta(y;a')\wedge \psi(x;y))$ belongs to $\eta(x)|_A$, so $\pi(x)|_A \cup \eta(x)|_A$ is already inconsistent.
	
Assume that the conclusion of the lemma does not hold, then there is $\phi(x;b)\in \pi(x)$, $\neg \phi(x;b)\in \eta(x)$. Let $N$ be maximal such that $\pi^{(N)}(x_{<N})|_A \cup \bigwedge_{i<N} \neg \phi(x_i;b)$ is consistent. Let $(\bar x^ i : i<\kappa)$ be a long Morley sequence in $\pi^{(N)}$ over $A$ such that for each $i$, there is $b_i\models \bigwedge_{j<N} \phi(x_j ^ i;y)$, $b_i \equiv_A b$. Now let $a'\models \pi(x)|_A \cup \eta(x)|_{Ab_{<\kappa}}$. Then $\neg \phi(a';b_i)$ holds for all $i$. But also for some $i$, $\bar x^i \models \pi^{(N)}|Aa'$, so $a'\hat{~}\bar x^ i$ satisfies $\pi^{(N+1)}$ over $A$. This contradicts the maximality of $N$.
\end{proof}

\begin{ex}
Consider the model companion of the theory of meet-trees in the language $\{\leq ,\wedge\}$ with an additional function $f$ from the main sort to an extra sort $C$ with no structure on it. This is an NIP ($\omega$-categorical) theory. Let $q(y)$ be the global type of a new element of $C$ and let $\pi_{\emptyset}(x)$ be the empty type of an element of the main sort. Then $q(y)$ and $\pi_{\emptyset}(x)$ are generically stable, but $q(y)\otimes \pi_{\emptyset}(x)$ is not. Also $\pi(y,x) = q(y)\cup \pi_{\emptyset}(x)$ is generically stable, but $\pi^{(2)}$ is not.

To see that $q(y)\otimes \pi_{\emptyset}(x)$ is not generically stable, consider a sequence $(c_i,a_i : i<\omega)$ such that:

-- $a_i\wedge a_j=a_i\wedge a_{j'}$, for all $j,j'>i$ and $a_i\wedge a_{i+1} < a_{i+1}\wedge a_{i+2}$;

-- $f(a_i \wedge a_{i+1})=c_i$;

-- $c_i$ satisfies $q$ over $a_{\leq i}$, $c_{<i}$.

\noindent
This is a Morley sequence of $q\otimes \pi_{\emptyset}$, but the sequence in the reverse order is not. Hence $q\otimes \pi_{\emptyset}$ does not satisfy (GS) by Proposition \ref{prop_gssym}. Also $\pi^ {(2)}$ is not generically stable as it implies $q\otimes \pi_{\emptyset}$ (when restricted to two of its variables).
\end{ex}

The following proposition will not be used later in the paper, but Proposition \ref{prop_isgs} in the proof of the main theorem is inspired from it.

\begin{prop}\label{prop_gsquant}
Let $\alpha(y)$ be a partial type, generically stable over $A$. Fix some $a,b\in \monster$, $b\models \alpha(y)|_A$ and let $\rho(x,y)\subseteq \tp(a,b/A)$. Then the partial type $\pi(x):= (\exists y) (\alpha(y) \wedge \rho(x,y))$ is generically stable over $A$.
\end{prop}
\begin{proof}
Note that for any set $B\supseteq A$, $\pi | B = (\exists y)(\alpha(y)|B\wedge \rho(x,y))$.

Since $\alpha(y)$ is $A$-invariant, $\pi(x)$ is also $A$-invariant. We first show that $\pi$ is Ind-definable using Lemma \ref{lem_def}. Fix a variable $\bar z$ and let $X_\alpha(y,\bar z)$ be the set of pairs $\{(b,\bar c) :b\models \alpha|A\bar c\}$. For any tuples $a$ and $\bar c$, we have $a\models \pi|A\bar c$ if and only if there is $b$ such that $\rho(a,b)$ and $(b,\bar c)\in X_\alpha$. As $X_\alpha$ is type-definable by Lemma \ref{lem_def}, this whole condition is type-definable. By the lemma again, $\pi$ is Ind-definable.

We next show (GS). Assume for a contradiction that for some $\phi(x;c)\in \pi$, the set $\pi^{(\omega)}((x_k:k<\omega)) \cup \{\neg \phi(x_k;c) : k<\omega\}$ is consistent. Let $(a_k)_{k<\omega}$ realize it. Note that if we replace $(a_k:k<\omega)$ by a sequence $(a'_k:k<\omega)$ which has the same type over $A$, then we can find $c'\equiv_A c$ such that $\neg \phi(a'_k;c')$ holds for all $k$. By invariance of $\pi$, we have $\phi(x;c')\in \pi$, so $(a'_k)$ also witnesses a failure of (GS).

We build by induction on $k$ tuples $(b_k:k<\omega)$ such that $\tp(a_k,b_k/A)=\tp(a,b/A)$ and $b_k\models \alpha | Aa_{<k}b_{<k}$. We can find $b_0$ since $a_0\models \pi|A$. Assume we have found $b_k$. As $a_{k+1}\models \pi|Aa_{\leq k}$, there is an automorphism $\sigma$ fixing $Aa_{\leq k}$ such that $\sigma(a_{k+1})\models \pi | Aa_{\leq k}b_{\leq k}$. By the remark above, we may replace the sequence $a_{>k}$ by $\sigma(a_{>k})$, since this does not alter the type of the full sequence $(a_i)_{i<\omega}$. Hence we may assume that actually $a_{k+1}\models \pi|Aa_{\leq k}b_{\leq k}$ and then we find $b_{k+1}$ as required.

We now have a sequence $(a_kb_k:k<\omega)$ such that $(a_k)_{k<\omega}\models \pi^{(\omega)}((x_k)_{k<\omega})$ and $c$ such that $\phi(x;c)\in \pi$ and $\neg \phi(a_k;c)$ holds for all $k$. Since the conditions $(a_k)_{k<\omega}\models \pi^{(\omega)}((x_k)_{k<\omega})$ and $b_k\models \alpha | Aa_{<k}b_{<k}$ are type definable by Lemma \ref{lem_def}, we can apply Ramsey and compactness and assume that the sequence $(a_kb_k:k<\omega)$ is indiscernible over $Ac$. Using (GS) for the type $\alpha$, we conclude that for every $k$, $b_k\models \alpha|Ac$. But by the definition of $\pi$, this means that $a_k\models \pi|Ac$. Contradiction.
\end{proof}

\section{Compressible types}\label{sec_compr}

In this section, we define compressible types. This notion was introduced in \cite{ExtDef2} (without giving it a name), where it is shown that a theory is distal if and only if all types are compressible. The reader may take this as a definition of distal theories.

\smallskip
If $A\subset \monster$ is any set of parameters, and $a\in \monster$ is a tuple, we let $(A,a)$ be the structure whose universe is $A$, with the induced structure coming from $a$-definable sets: for every $\phi(\bar x;a)\in L(A)$, we have a predicate $R_\phi(\bar x)$ interpreted as $\{ \bar b\in A : \monster \models \phi(\bar b;a)\}$. If $\mathcal M \equiv (A,a)$, then it is isomorphic to $(A',a)$ for some $A'\subset \monster$.

We think of $(A,a)$ as a first order structure encoding the type of $a$ over $A$ and we will be mainly considering properties of $\tp(a/A)$ that translate into first order properties of the structure $(A,a)$. For example, if $\phi(x;y)\in L$, the fact that $\tp_\phi(a/A)$ is definable is a first order property of $(A,a)$ in the sense that if $(A',a)\equiv (A,a)$, then $\tp_\phi(a/A)$ is definable if and only if $\tp_\phi(a/A')$ is definable.

\begin{defi}\label{def_compr}
A type $p(x)=\tp(a/A)$ is \emph{compressible} if given an $|A|^+$-saturated elementary extension $(A,a)\prec (A',a)$, for any formula $\phi(x;y)\in L$, there is some $\zeta(x;e)\in \tp(a/A')$ such that $\zeta(x;e)\vdash \tp_{\phi}(a/A)$.
\end{defi}

Observe that by compactness, this definition is equivalent to the following: for any formula $\phi(x;y)$, there is a formula $\zeta(x;t)$ such that for any finite $A_0\subseteq A$, there is $e\in A$ such that $a\models \zeta(x;e)$ and $\zeta(x;e)\vdash \tp_{\phi}(a/A_0)$.

Recall the notion of honest definitions from \cite{DepPairs} (see also \cite[Chapter 3]{NIPBook}): Given $(A,a)$ and an NIP formula $\phi(x;y)$, there is $(A,a)\prec (A',a)$ and some $\theta_1(y;e)\in L(A')$ such that $\theta_1(A;e)=\phi(a;A)$ and $\theta_1(A';e)\subseteq \phi(a;A')$. We call $\theta_1(y;e)$ an honest definition of $\phi(a;y)$ over $A$. Note that $e$ can be found in any $|A|^+$-saturated extension of $(A,a)$ since the requirements on it are first order expressible over $A$ in that structure.

One easily checks that if $\tp(a/A)$ is compressible and $\zeta(x;e)$ is as in Definition \ref{def_compr}, then the formula \[\theta_1(y;e) \equiv (\forall x) \left [\zeta(x;e) \impl \phi(x;y)\right ]\]
is an honest definition of $\phi(a;y)$ over $A$. In fact, we even have the stronger property $\theta_1(\monster;e) \subseteq \phi(a;\monster)$.

\begin{lemme}\label{lem_equivcompr}
	If $(A,a)\equiv (A',a')$, then $\tp(a/A)$ is compressible if and only if $\tp(a'/A')$ is compressible.
\end{lemme}
\begin{proof}
Assume that $\tp(a/A)$ is compressible. Fix a formula $\phi(x;y)$ and let $\zeta(x;t)$ be given by compressibility of $\tp(a/A)$. Define also \[\theta(y;t) \equiv \left [ (\forall x) \zeta(x;t) \impl \phi(x;y)\right ] \vee \left [(\forall x) \zeta(x;t)\impl \neg \phi(x;y)\right ].\] By compressibility, for any finite $A_0\subseteq A$, there is $e\in A$ such that $\zeta(a;e)$ holds and $A_0\subseteq \theta(A;e)$. Hence for any integer $m$, 
\[ (A,a)\models (\forall y_0,\ldots,y_{m-1})(\exists t)\left [\zeta(a,t)\wedge \bigwedge_{i<m} \theta(y_i;t)\right ]. \]

Since $(A',a')$ is elementarily equivalent to $(A,a)$, it satisfies all those formulas as $\phi$ varies. This in turns implies that $\tp(a'/A')$ is compressible. (Note that $\theta(y;t)$ says that $\zeta(x;t)$ implies a $\phi$-type over $y$. If both $\zeta(a;e)$ and $\theta(b;e)$ hold, then the $\phi$-type over $b$ implied by $\zeta(x;e)$ has to be that of $a$ since $a\models \zeta(x;e)$.)
%
\end{proof}

The following was implicit in the proof of \cite[Proposition 19]{ExtDef2}. Recall that two types $p(x)$ and $q(y)$ over the same set $A$ are \emph{weakly orthogonal} if $p(x)\cup q(y)$ implies a complete type over $A$.

\begin{lemme}\label{lem_equivalentcompr}
Let $p(x) = \tp(a/A)$ be any type and take $(A,a)\prec (A',a)$, $|A|^+$-saturated. Then the following are equivalent:
\begin{enumerate}
	\item $p$ is compressible;
	\item  $\tp(a/A')$ is weakly orthogonal to all types $q(y) \in S(A')$ finitely satisfiable in $A$;
	\item for any $q(y) \in S(A')$ finitely satisfiable in $A$, $\tp_x(a/A')\cup q(y)$ implies a complete type  in variables $x\hat{~}y$ over $\emptyset$.
\end{enumerate}
\end{lemme}
\begin{proof}
Assume first that $p(x)$ is compressible. Let $q(y)\in S(A')$ be finitely satisfiable in $A$ and let $\phi(x;y)\in L$. Let $\zeta(x;e)$ be given by the definition of compressibility. Consider the formula $$\theta(y;e) \equiv \left [ (\forall x) \zeta(x;e) \impl \phi(x;y)\right ] \vee \left [(\forall x) \zeta(x;e)\impl \neg \phi(x;y)\right ].$$
By assumption, every $b\in A$ satisfies $\theta(y;e)$. By finite satisfiability of $q$, we have $q(y)\vdash \theta(y;e)$. Thus for some $\epsilon \in \{0,1\}$, $\zeta(x;e)\wedge q(y) \vdash \phi(x;y)^{\epsilon}$ and in particular $p(x)\wedge q(y)\vdash \phi(x;y)^{\epsilon}$. This shows that (3) holds.

To see that (2) also holds, take $b\in A'$ a finite tuple. There is $A\subseteq A'' \subseteq A'$ such that $b\in A''$, $(A,a)\prec (A'',a)\prec (A',a)$ and $|A''|=|A|$. By Lemma \ref{lem_equivcompr}, $\tp(a/A'')$ is also compressible. Hence everything done for $A$ also applies for $A''$. Consider the type $q'(y\hat{~}z) = q(y)\cup \{z=b\}$. Then $q'$ is finitely satisfiable in $A''$. By the previous paragraph, $\tp_x(a/A')\cup q'(y)$ implies a complete type over $\emptyset$. As $b\in A'$ was arbitrary, this implies that $\tp(a/A')$ and $q(y)$ are weakly orthogonal.

Assume now that that $p'(x):=\tp(a/A')$ is weakly orthogonal to every $q(y)\in S(A')$ finitely satisfiable in $A$ and take some formula $\phi(x;y)$. Let $S\subseteq S_y(A')$ be the set of types finitely satisfiable in $A$. It is a closed subset of $S_y(A')$ and thus compact. For each $q\in S$, for some $\epsilon_q \in \{0,1\}$, we have $p'(x)\wedge q(y) \vdash \phi(x;y)^{\epsilon_q}$. By compactness, there are formulas $\zeta_q(x)\in p'(x)$ and $\theta_q(y)\in q$ such that already $\zeta_q(x) \wedge \theta_q(y) \vdash \phi(x;y)^{\epsilon_q}$. Let $T\subseteq S$ finite such that the family $\{\theta_q(y):q\in T\}$ covers $S$. Define $\zeta(x) = \bigwedge_{q\in T} \zeta_q(x)$. Observing that any $b\in A$ satisfies $\bigvee_{q\in T} \theta_q(y)$, one sees that $\zeta(x)\vdash \tp_{\phi}(a/A)$.
\end{proof}

\section{Decomposition}\label{sec_main}

We now come to the main theorem of this paper. Intuitively, it says that if $T$ is NIP, and $p(x)=\tp(a/A)$ is any type, then there is a generically stable partial type $\pi(x)$ contained in $p$ and such that $p$ is compressible \emph{up to} $\pi$.

\begin{thm}\label{th_main}($T$ is NIP.)
Let $p(x)=\tp(a/A)$ be any type. Then there is $\pi(x)$ generically stable over $A$ with $\pi(x)|_A \subseteq p(x)$, such that if $(A,a)\prec (A',a)$ is $|A|^+$-saturated and $q(y)$ is a global type finitely satisfiable in $A$, then $\tp_x(a/A')\cup (\pi\otimes q)|_{A'}(x,y)$ implies the complete type $(q\otimes p)(y,x)|_A$.
\end{thm}

Note that as $\pi$ is Ind-definable, we automatically have $\pi(x)|_{A'} \subseteq p'(x)$. Also, if one prefers to think about the dual ideal $I_\pi$ rather than the partial type $\pi$, then the conclusion can be rephrased by saying that any two $I_\pi$-wide extensions of $p'(x)$ to a realization $b$ of $q(y)$ have the same restriction to $Ab$.

\medskip
We proceed with the proof.

Let $p(x) = \tp(a/A)$. Let $q(y)$ be a global $A$-finitely satisfiable type which will be fixed for most of the proof. We will write $\bar a\models \Mor(q)|A$ to mean that $\bar a$ is a Morley sequence of $q$ over $A$. The indiscernible sequences we consider will always be implicitly assumed to be indexed by a dense order without endpoints.

Let $\Omega$ be the class of types $\tp(\bar a/Aa)$, where $\bar a=(a_i:i\in \mathcal I)$ is an indiscernible sequence and there is a Morley sequence $I$ of $q$ over $Aa$ such that $I+\bar a$ is $A$-indiscernible (hence is a Morley sequence of $q$ over $A$).

\begin{lemme}
There is $s(\bar x)\in \Omega$ such that if $\bar a\models s$ and $\bar a \unlhd_{A} \bar a'$ with $\tp(\bar a'/Aa)\in \Omega$, then $\bar a \unlhd_{Aa} \bar a'$.
\end{lemme}
\begin{proof}
The proof is similar to that of Proposition \ref{prop_existssbase}: if some $\bar a$ does not have the required property, we increase it introducing some additional alternation over $Aa$ and iterate. So start with any $s_0 \in \Omega$ and $\bar a_0 \models s_0$. We try to build by induction an increasing sequence of types $s_i \in \Omega$, $i<(|A|+|T|)^+$ (where increasing implies that the variables of $s_i$ are included in that of $s_j$, $i<j$) as follows:

At a limit stage $\lambda$, set $s_\lambda = \bigcup_{\alpha<\lambda} s_\alpha$. For each $\alpha<\lambda$, let $\bar a_\alpha \models s_\alpha$ and let $I_\alpha$ be a Morley sequence of $q$ over $Aa$ such that $I_\alpha +\bar a_\alpha$ is $A$-indiscernible. Pick an ultrafilter $\mathfrak F$ on $\lambda$ extending the cofinal filter. Let $(I_\lambda,\bar a_\lambda)$ realize the limit of $(\tp(I_\alpha,\bar a_\alpha/Aa):\alpha<\lambda)$ along $\mathfrak F$. Then $\bar a_\lambda$ realizes $s_\lambda$, $I_\lambda$ is a Morley sequence of $q$ over $A$ and $I_\lambda+\bar a_\lambda$ is $A$-indiscernible. Hence $s_\lambda \in \Omega$. Note also that for each $\phi$, $\textsf T(\bar a_\lambda,\phi) = \sup_{\alpha<\lambda} \textsf T(\bar a_\alpha,\phi)$. (Recall the definition of $\textsf T$ given in Section \ref{sec_prel}.)

Assume that $s_\alpha$ has been built and let $\bar a_\alpha \models s_\alpha$. If $s_\alpha$ has the required property, we are done. Otherwise, there is some $\bar a_\alpha \unlhd_{A} \bar a_{\alpha+1}$ such that $\tp(\bar a_{\alpha+1}/Aa)\in \Omega$ and $\bar a_\alpha \ntrianglelefteq_{Aa} \bar a_{\alpha+1}$. This implies that for some formula $\phi\in L(Aa)$, we have $\textsf T(\bar a_{\alpha+1},\phi)>\textsf T(\bar a_\alpha,\phi)$. Set $s_{\alpha+1}=\tp(\bar a_{\alpha+1}/Aa)$.

As the numbers $\textsf T(\cdot, \phi)$ must remain finite, this construction ends after less than $(|A|+|T|)^+$ steps.
\end{proof}

Pick some $s\in \Omega$ given by the lemma and let $\bar a_*$ realize it. One can think of $\bar a_*$ as encoding the \emph{$q$-stable part} of $a$. Let also $I$ be a Morley sequence of $q$ over $Aa$ (indexed by $\mathbb Q$) such that $I+\bar a_*$ is $A$-indiscernible. We show the following properties:

\begin{itemize}
\item[$\boxtimes_0$] $I \models \Mor(q)|Aa$ and $I+\bar a_*\models \Mor(q)|A$.

\item[$\boxtimes_1$] Whenever $I+\bar a_* \unlhd_{A} I+I_0+\bar a_* +I_1$, then $\bar a_* \unlhd_{Aa} I+I_0+\bar a_*+I_1$, $I_0+I_1 \models \Mor(q)|AIa$ and $I_1 \models \Mor(q)|AaII_0\bar a_*$.

\end{itemize}
Property $\boxtimes_0$ is immediate from the construction. We show property $\boxtimes_1$. Let $I+\bar a_* \unlhd_{A} I+I_0+\bar a_* +I_1$. Cut $I$ into two infinite pieces as $I=J_0+J_1$ and notice that $\tp(J_1+I_0+\bar a_*+I_1/Aa)\in \Omega$ as witnessed by $J_0$. We also have $\bar a_* \unlhd_A J_1+I_0+\bar a_*+I_1$ since both those sequences are indiscernible over $A$. Hence by the choice of $\bar a_*$, we have $\bar a_* \unlhd_{Aa} J_1+I_0+\bar a_*+I_1$. Since $J_1$ can be an arbitrarily large subset of $I$, we have $\bar a_*\unlhd_{Aa} I+I_0+\bar a_*+I_1$.

For the second point, let $\op(I_{-1})$ realize $\lim(I)$ over everything and $I_2$ realize $\Mor(q)$ over everything, including $I_{-1}$. Then $I+I_{-1}+I_0+\bar a_*+I_1+I_2$ is $A$-indiscernible, or in other words $I+\bar a_* \unlhd_{A} I+I_{-1}+I_0+\bar a_*+I_1+I_2$. We therefore have $\bar a_* \unlhd_{Aa} I+I_{-1}+I_0+\bar a_*+I_1+I_2$ by the previous paragraph, which implies that $I_{-1}+I_0$ and $I_1+I_2$ are mutually indiscernible over $AIa\bar a_*$. Since $I_{-1}+I_2$ is a Morley sequence of $q$ over $AIa$, so is $I_0+I_1$ and similarly since $I_2$ is a Morley sequence of $q$ over $AaII_0 \bar a_*$, so is $I_1$.


\smallskip
Let $\pi(x)=\pi_{q}(x)$ be the partial type \[p(x) \cup (\exists \bar a'_*)(\tp(x\bar a'_*/A)=\tp(a\bar a_*/A) ~\&~ \bar a'_* \text{ is indiscernible over }\monster).\]

\begin{lemme}\label{lem_b}
If $B$ is any set of parameters containing $A$, then $\pi(x)|_B$ is: 
\[p(x) \cup (\exists \bar a'_*)(\tp(x\bar a'_*/A)=\tp(a\bar a_*/A) ~\&~ \bar a'_*\text{ is indiscernible over }B).\]
\end{lemme}
\begin{proof}
It is clear that this partial type is included in $\pi|_B$. Conversely, assume that $\tp(a' \bar a'_*/A)=\tp(a\bar a_*/A)$ and $\bar a'_*$ is indiscernible over $B$. We have to show that $a'\models \pi|_B$, {\it i.e.}, that $\tp(a'/B)$ has a global extension which satisfies $\pi$. Work in a larger monster model $\monster' \succ \monster$. By Ramsey and compactness, build an indiscernible sequence $\bar a''_*$ over $\monster$ having the same type as $\bar a'_*$ over $B$. Then we can find $a''$ such that $\tp(a''\bar a''_*/B)=\tp(a'\bar a'_*/B)$. This shows that $a''\models \pi|\monster$ and hence $a'\models \pi|B$.
\end{proof}

In particular, taking $B=A$, this shows that $\pi(x)$ is consistent.

For later purposes, let us note now that if we have any number of types of the form above, with the same $a$ and different $\bar a_*$, then their conjunction is also consistent. In fact, we can concatenate all the $\bar a_*$'s into one indiscernible sequence of possibly infinite tuples and apply the same argument.

\begin{lemme}\label{lem_afterrem}
The partial type $\pi(x)$ is Ind-definable over $A$.
\end{lemme}
\begin{proof}
Given any index set $\mathcal I$, the set $\{(\bar a=(a_i)_{i\in \mathcal I},b) : \bar a $ is indiscernible over $Ab\}$ is type definable over $A$. The lemma then follows at once from lemmas \ref{lem_b} and \ref{lem_def}.
\end{proof}

Let $\bar a_* \unlhd_{AIa} I_0+\bar a_*+I_1$. Then $I+\bar a_* \unlhd_{Aa} I+I_0+\bar a_*+I_1$, so by $\boxtimes_1$, we have $I+I_0+I_1 \models \Mor(q)|Aa$.

Next build $\bar b_*$ such that:

\begin{itemize}
\item[$\otimes_0$] $I+I_0+I_1+\bar b_* \models \Mor(q)|A$;

\item[$\otimes_1$] $\tp(\bar b_*/Aa)= \tp(\bar a_*/Aa)$.
\end{itemize}

To see that this is possible, construct $I_0+\bar a_* \unlhd_{AIa} I_0+I_1'+\bar a_*$, where $I_1'$ has the same order type as $I_1$. Then $I+I_0+I_1' \equiv_{Aa} I+I_0+I_1$, as both are Morley sequences of $q$ over $Aa$ by $\boxtimes_1$, and we can take $\bar b_*$ such that $I+I_0+I_1'+\bar a_* \equiv_{Aa} I+I_0+I_1+\bar b_*$.

By Proposition \ref{prop_existssbase}, we can now find $I_0+\bar a_*+I_1\unlhd_{AI} J_0+\bar a_0+\bar a_*+\bar a_1+J_1$ such that, denoting $\bar a_0+\bar a_*+\bar a_1$ by $\bar a_{**}$:

\begin{itemize}
\item[$\otimes_2$] $\bar b_* \models \Mor(q)|AIJ_0J_1$;

\item[$\otimes_3$] $\bar a_{**}$ strongly dominates $\bar b_*$ over $(J_0+J_1,AI)$.
\end{itemize}


\smallskip
We now come to the main technical lemma of this proof.

\begin{lemme}\label{lem_maintech}
Assume that $a'\models \pi|B$, then there is $\bar b'_*$ such that $\tp(a'\bar b'_*/A)= \tp(a\bar a_*/A)(=\tp(a\bar b_*/A))$ and $\bar b'_*\models \Mor(q)|B$.
\end{lemme}
\begin{proof}
	We have \[I+\bar a_* \unlhd_{A} I+I_0+\bar a_*+I_1 \unlhd_{A} I+J_0+\bar a_0+\bar a_*+\bar a_1+J_1.\] By transitivity of $\unlhd_{A}$ and $\boxtimes_1$, we have:
	
	\begin{itemize}
		\item[$\otimes_4$] $\bar a_* \unlhd_{Aa} I+J_0+\bar a_0+\bar a_*+\bar a_1+J_1$.
		
	\end{itemize}
	
	Let $\op(J_2)$ be a Morley sequence in $\lim(J_1)$ over everything constructed so far. We then have \[I+J_0+\bar a_0+\bar a_*+\bar a_1+J_1\unlhd_{A} I+J_0+\bar a_0+\bar a_*+\bar a_1+J_1+J_2.\] By transitivity of $\unlhd_A$ and the last statement in $\boxtimes_1$, $J_2$ is a Morley sequence of $q$ over $AaIJ_0 \bar a_{**} J_1$. Also we have $J_0+J_1 \unlhd J_0+J_1+J_2$ over everything constructed so far, and by the definition of strong domination, $\bar a_{**}$ strongly dominates $\bar b_*$ over $(J_0+J_1+J_2,AI)$.
	
	Also by $\otimes_2$ and the construction of $J_2$:
	\begin{itemize}
		\item[$\otimes_5$] $\bar b_* \models \Mor(q)|AIJ_0J_1J_2$.
	\end{itemize}

Assume that $a'\models \pi|B$ and by Lemma \ref{lem_b}, let $\bar a'_*$ be such that $\tp(a'\bar a'_*/A)=\tp(a\bar a_*/A)$ and $\bar a'_*$ is indiscernible over $B$.
Construct $\bar a'_* \unlhd_{Ba'} I'+J_0'+\bar a_0'+\bar a'_*+\bar a'_1+J'_1$, where each primed sequence has the same order type as its unprimed counterpart. By $\otimes_4$, we have:
$$\tp(I'J'_0\bar a'_0\bar a'_*\bar a'_1J'_1\hat{~} a'/A)=\tp(IJ_0\bar a_0\bar a_*\bar a_1J_1\hat{~} a/A).$$
Build then $J'_2\models \Mor(q)$ over everything constructed so far. Then again
$$\tp(I'J'_0\bar a'_0\bar a'_*\bar a'_1J'_1J'_2\hat{~} a'/A)=\tp(IJ_0\bar a_0\bar a_*\bar a_1J_1J_2\hat{~} a/A).$$
So we can find some $\bar b'_*$ such that
$$\tp(I'J'_0\bar a'_0\bar a'_*\bar a'_1J'_1J'_2\bar b'_*\hat{~} a'/A)=\tp(IJ_0\bar a_0\bar a_*\bar a_1J_1J_2\bar b_*\hat{~} a/A).$$
Let $\bar a'_{**} = \bar a'_0\bar a'_*\bar a'_1$. We now have:
\begin{itemize}
\item[$\bullet_0$] $J'_0+\bar a'_{**}+J'_1$ is indiscernible over $BI'$;

\item[$\bullet_1$] $J'_2$ is a Morley sequence of $q$ over $B\cup J'_0\bar a'_{**}J'_1$;

\item[$\bullet_2$] $\bar a'_{**}$ strongly dominates $\bar b'_*$ over $(J'_0+J'_1+J'_2, AI')$;

\item[$\bullet_3$] $\bar b'_* \models \Mor(q)|AI'J'_0J'_1J'_2$.
\end{itemize}
By Fact \ref{prop_extstablebase} (taking $d=B$, $A=AI'$,  and $I=J'_0+J'_1+J'_2$), $\bar b'_*$ is a Morley sequence of $q$ over $B$ as required.
\end{proof}

We can now show:

\begin{prop}\label{prop_isgs}
The partial type $\pi(x)$ is generically stable over $A$.
\end{prop}
\begin{proof}
We have already seen that $\pi(x)$ is Ind-definable over $A$.

To show property (GS), let $\lambda$ be some large enough cardinal.  Assume that for some $\phi(x;d)\in \pi$, the set $\pi^{ (\omega)}((x_i)_{i<\omega})|A \cup \{\neg \phi(x_i;d):i<\omega\}$ is consistent, then by compactness, so is $\pi^{ (\lambda)}((x_i)_{i<\lambda})|A \cup \{\neg \phi(x_i;d):i<\lambda\}$. Let $(a_i:i<\lambda)$ realize it. We build inductively a sequence $(\bar a_*^ i : i<\lambda)$ such that for each $i$, $\tp(a_i\bar a_*^ i/A)=\tp(a\bar a_*/A)$, $a_i\models \pi|A{a_{<i}\bar a_*^{<i}}$, and the sequence $\bar a_*^0+\bar a_*^1 + \cdots$ is a Morley sequence of $q$ over $A$.

Let $k<\lambda$ and assume that we have constructed $\bar a_*^{l}$ for $l<k$. As $a_k \models \pi|Aa_{<k}\bar a_*^{<k}$, there is $\bar c_*^k$ such that $\tp(a_k\bar c_*^k/A)=\tp(a\bar a_*/A)$ and $\bar c_*^k\models \Mor(q)|Aa_{<k}\bar a_*^{<k}$. The sequence $(a_i : k<i<\lambda)$ realizes $\pi^{(\lambda)}$ over $C:=Aa_{\leq k}\bar a_*^{<k}$. Therefore $\tp((a_i)_{k<i<\lambda})/ C)$ has an extension to $C\bar c_*^k$ which contains $\pi^{(\lambda)}|C\bar c_*^k$. Let $(a'_i:k<i<\lambda)$ realize that extension. There is an automorphism $\sigma$ fixing $C$ pointwise and sending $a_i$ to $a'_i$ for all $k<i<\lambda$. Set $\bar a_*^{k} = \sigma^{-1}(\bar c_*^k)$. Then $\tp(a_k\bar a_*^k/A)=\tp(a\bar a_*/A)$, $\bar a_*^k\models \Mor(q)|C$ and $(a_i:k<i<\lambda)$ realizes $\pi^{ (\lambda)}|C\bar a_*^k$. This finishes the construction. Note that at a $\eta<\lambda$ limit the construction can go on: we have that $(a_i:\eta\leq i <\lambda)$ realizes $\pi^{ (\lambda)}$ over $a_{<\eta}\bar a_*^{<\eta}$ since this is true over any finite subset of it.



Having done this, recall our parameter $d$ from the first paragraph. By shrinking of indiscernibles, for some $i<\lambda$, the sequence $\bar a^ i_*$ is indiscernible over $Ad$ which implies that $a_i\models \pi|Ad$. This contradicts the hypothesis that $a_i \models \neg \phi(x;d)$. Hence $\pi(x)$ is generically stable.
\end{proof}

Next, we show domination.

\begin{prop}\label{prop_maindom}
Let $I$ be a Morley sequence of $q$ over $Aa$ and let $b\models q|AI$. Assume that $a\models \pi|AIb$, then $b\models q|Aa$.
\end{prop}
\begin{proof}
Let $\bar b\models \Mor(q)|AI$ containing $b$ as one of the elements in the sequence. Then $\tp(a/AIb)\cup \pi|AI\bar b$ is consistent. Let $a'$ realizes that type and let $\sigma$ be an automorphism fixing pointwise $AIb$ and sending $a$ to $a'$. Then replacing $\bar b$ by $\sigma^{-1}(\bar b)$, we may assume that $a\models \pi|AI\bar b$.
	
As $a\models \pi |AI\bar b$, by Lemma \ref{lem_maintech} there is $\bar a'_* \models \Mor(q)|AI\bar b$ such that $\tp(a\bar a'_*/A)=\tp(a\bar a_*/A)$. Then the sequence $I+\bar b+\bar a'_*$ is indiscernible over $A$. This implies that $I+\bar a'_* \unlhd_{A} I+\bar b+\bar a'_*$ and $\tp(\bar b+\bar a'_*/Aa)\in \Omega$. Hence by $\boxtimes_1$, $I+\bar b$ is indiscernible over $Aa$. Therefore $\bar b\models \Mor(q)|Aa$ and in particular, $b\models q|Aa$.
\end{proof}

We now have all we need to conclude. For any type $q(y)$, the construction described above supplies us with a generically stable type $\pi_{q}(x)$. Let $\pi(x)$ be the conjunction of all those types as $q$ varies. Then by the remark before Lemma \ref{lem_afterrem}, $\pi(x)$ is consistent, and therefore generically stable. Let $(A,a)\prec (A',a)$ be $|A|^+$-saturated. So $A'$ contains Morley sequences over $Aa$ of all types $q$ finitely satisfiable in $A$. The domination property then follows from Proposition \ref{prop_maindom}, and the theorem is proved.

\smallskip
As a corollary, we obtain a more explicit form of honest definitions. Recall the following special case of the $(p,q)$-theorem of Matou\u sek (see \cite[Section 6.2]{NIPBook}), which will be needed to prove uniformity.

\begin{fait}\label{fact_pq}
Let $\theta(y;t)$ be an NIP formula. Then for some $n$ and $N$ the following holds:

If $B\subset \monster^{|t|}$ is any set of parameters and $A\subset \monster^{|y|}$ finite such that for any $A_0\subseteq A$ of size $\leq n$, there is $e\in B$ with $A_0\subseteq \theta(A;e)$, \underline{then} there is a finite set $B_* \subseteq B$ of size $\leq N$ such that for any $a\in A$, for some $e\in B_*$, $\theta(a;e)$ holds.

\end{fait}

\begin{cor}\label{cor_betterhonest}($T$ is NIP.)
Let $\phi(x;y)$ be any formula, then there are formulas $\zeta(x;t)$, $\theta_0(y;t)$, $\theta_1(y;t)$ and $\delta(x,y,t;u)$, such that the following holds:

For any small set $A$ of size $\geq 2$ and $a\in \monster^{|x|}$, there is $d\in A^{|u|}$ such that:

-- for all $(b,e)\in A^{|y|+|t|}$, $a\models \delta(x,b,e;d)$;

-- for any finite $A_0\subseteq A$, there is some $e\in A$, $a\models \zeta(x;e)$, $A_0\subseteq \theta_0(A;e)\cup \theta_1(A;e)$ and for $\epsilon=0,1$ we have:
\[\models (\forall x,y)\left [\theta_\epsilon(y;e)\wedge \zeta(x;e) \wedge \delta(x,y,e;d) \right ]\impl \phi(x;y)^\epsilon. \]
\end{cor}
\begin{proof}
Let $\phi(x;y)\in L$ and for now fix some type $p(x)=\tp(a/A)$.

Let $\pi(x)$ be given by Theorem \ref{th_main} for this $p$. Let $(A,a)\prec (A',a)$ be sufficiently saturated and let $S$ be the set of types in $S_y(A')$ finitely satisfiable in $A$. For $q(y)\in S$, let $\epsilon = \epsilon(q)$ be such that $q(y)\otimes p(x) \vdash \phi(x;y)^{\epsilon}$. By Theorem \ref{th_main} and compactness there are: 

-- $\theta_q(y;e)\in q|A'$,

-- $\zeta_q(x;e)\in \tp(a/A')$,

-- finitely many pairs $\psi_{q,i}(x;y,t), d\psi_{q,i}(y,t;d)$, $d\in A$, where each one of the partial types $(\psi_{q,i}(x;y,t),d\psi_{q,i}(y,t;d))$ is in $\pi(x)$, such that 
\[\theta_q(y;e)\wedge \zeta_q(x;e)\wedge \bigwedge_i \left ( d\psi_{q,i}(y,e;d)\impl \psi_{q,i}(x,y,e)\right ) \vdash \phi(x;y)^{\epsilon}.\] Allowing $e$ and $d$ to be infinite, we can assume that those parameters are the same for all $q$.

By compactness, let $S_*\subseteq S$ be a finite set such that $\{\theta_q(y;e) : q\in S_*\}$ covers $S$. For $\epsilon = 0,1$, set $S_\epsilon = \{q\in S_*: \epsilon(q)=\epsilon\}$. Define $\theta_\epsilon(y;e)= \bigvee_{q\in S_\epsilon} \theta_q(x;e)$, $\zeta(x;e)=\bigwedge_{q\in S_*} \zeta_q(x;e)$ and let $(\psi_i(x,y,t),d\psi_i(y,t;d))_{i<n}$ enumerate all pairs of formulas $(\psi_{q,i}(x;y,t),d\psi_{q,i}(y,t;d))$ for $q\in S_*$. We can now revert to assuming that $e$ and $d$ are finite. Then we have:
\[\models (\forall x,y)\left [\theta_\epsilon(y;e)\wedge \zeta(x;e) \wedge \bigwedge_{i<n}\left  ( d\psi_i(y,e;d)\impl \psi_i(x,y,e)\right )\right ]\impl \phi(x;y)^\epsilon. \]
We obtain what we want by setting $\delta(x,y,t;d) = \bigwedge_{i<n}\left  ( d\psi_i(y,t;d)\impl \psi_i(x,y,t)\right )$.

It remains to show uniformity, {\it i.e.}, that the formulas $\zeta$, $\theta_\epsilon$, $\delta$ can be chosen so as to depend only on $\phi$ and not on $p$.
The proof is exactly like the proof of uniformity of honest definitions in \cite[Theorem 11]{ExtDef2}. Fix some $\phi(x;y)$ and $n<\omega$. Extend the language $L$ to add a new unary relation symbol $\pred(y_0)$. Let $L_\pred$ be the resulting language. Let $M\models T$ and $M'$ an expansion of $M$ to $L_\pred$. Set $A=\pred(M')$ and pick some $a\in M$. Then there are $\zeta(x;t)$, $\delta(x,y,t;u)$, $\theta_\epsilon(y;t)$ and $d\in A^{|u|}$ such that for any $A_0\subseteq A$ \emph{of size $\leq n$}, we can find $e\in A$ as in the statement. This last condition is first-order expressible as we are quantifying over subsets of $A$ of size $\leq n$. Hence by compactness, there are finitely many tuples $\{ (\zeta_i,\delta_i,\theta_{\epsilon,i}) : i<r_*\}$ such that for any $M\models T$, any expansion $M'$ of $M$ to $L_\pred$ and any $a\in M$, there is one tuple of formulas in this finite set which has the property above (still quantifying over $A_0$ of size $\leq n$). By usual coding techniques, we can find one tuple of formulas $(\zeta,\delta,\theta_0,\theta_1)$ which works for all expansions to $L_\pred$ of a model of $T$ and all choices of $a$ and $A$, as long as $|A|\geq 2$.

Take $n$ large enough so that Fact \ref{fact_pq}, applies for the formula $\theta(y;t):=\theta_0(y;t)\vee \theta_1(y;t)$. Then take any $(A,a)$, $A$ of size $\geq 2$, and any $A_0\subseteq A$ finite. For any $A_1\subseteq A_0$ of size $\leq n$, we can find $e\in A$ such that $A_1\subseteq \theta(A;e)$. By Fact \ref{fact_pq}, we can find $e_0,\ldots ,e_{N-1}\in A$ such that \[A_0 \subseteq \bigcup_{k<N} \theta(A;e_k) = \bigcup_{k<N} \theta_0(A;e_k)\cup \theta_1(A;e_k).\]  We now define: 

$\tilde \zeta(x;t_0\ldots t_{N-1}) = \bigwedge_{k<N} \zeta(x;t_k)$,

$\tilde \theta_\epsilon(y;t_0,\ldots,t_{N-1}) = \bigvee_{k<N} \theta_\epsilon(y;t_k)$ and

$\tilde \delta(x,y,t_0\ldots t_{N-1};d_0\ldots d_{N-1})=\bigwedge_{k<N} \delta(x,y,t_k;d_k)$.

\noindent
Those formulas have the required properties.
\end{proof}

Taking the notations of the corollary, note that if $(A,a)\prec (A',a)$ is sufficiently saturated, we can find $e\in A'$ such that $A\subseteq \theta_0(A',e)\cup \theta_1(A',e)$ and $a\models \zeta(x;e)$. Then $\theta_1(y;e)$ is an honest definition of $\phi(a;y)$ over $A$ since by elementarity, for all $(b,e')\in A'^{|y|+|t|}$, $a\models \delta(x,b,e';d)$.

\begin{cor}
	In Theorem \ref{th_main}, we can take $\pi$ to be based on some $A_0\subseteq A$ of size $\leq |T|$.
\end{cor}
\begin{proof}
	In Theorem \ref{th_main}, we can replace $\pi(x)$ by any $\pi_0(x)\subseteq \pi(x)$ which is generically stable and contains the $|T|$ many partial types $(\psi_i(x;y,t),d\psi_i(y,t;u))$ defined in the proof of Corollary \ref{cor_betterhonest} as $\phi(x;y)$ ranges over $L$. Such a $\pi_0$ exists by Lemma \ref{lem_reducepi}.
\end{proof}

\subsection{Existence of compressible types}

It is an open question whether any unstable NIP theory has a distal (non-constant) indiscernible sequence (as defined in \cite{distal}). In this section, we answer positively a related question, namely we construct a non-realized compressible type over a model.

If $p(x)$ is a type over some set $A$ and $\phi(x;y)$ a formula, we denote by $p_\phi$ the partial type of all instances of $\phi$ and $\neg \phi$ in $p$. We say that $p_\phi$ is definable if there is a formula $d\phi(y)\in L(A)$ such that for any $b\in A$, $p\vdash \phi(x;b) \iff \models d\phi(b)$.

Let $M\prec N$, $q\in S_x(N)$ and $p=q|_M$. We say that $q$ is a conservative extension of $p$ if for any formula $\phi(x;y)\in L(M)$, if $q_\phi$ is definable, then so is $p_\phi$. In particular, if $p$ is not a definable type, then $q$ is not either and so is not a realized type.

\begin{lemme}
Let $(p_i:i<\alpha)$ be an increasing sequence of conservative extensions of $p$, then $p_*:=\bigcup_{i<\alpha} p_i$ is also a conservative extension of $p$.
\end{lemme}
\begin{proof}
We can assume that $\alpha$ has no last element. Let $\phi(x;y)$ be a formula such that $(p_*)_{\phi}$ is definable by a formula $d\phi(y)$. Then there is some $i<\alpha$ for which the parameters in $d\phi$ belong to the domain of $p_i$. But then $(p_i)_\phi$ is also defined by the formula $d\phi(y)$ and therefore by hypothesis $p_\phi$ is definable.
\end{proof}

\begin{lemme}\label{lem_skolem}
Let $\kappa \geq |T|$. Let $M$ be $\kappa^+$-saturated and $q\in S(M)$ which is finitely satisfiable in a set of size $\kappa$. Then there is an extension $M\prec N$ containing a realization of $q$ such that for any finite tuple $c\in N$, $\tp(c/M)$ is finitely satisfiable in some subset of size $\kappa$.
\end{lemme}
\begin{proof}
Expand $M$ to a Skolemization $T^{Sk}$ of $T$. Let $A\subset M$ be of size $\kappa$ such that $q$ is finitely satisfiable in $A$. Extend $q$ to a type $\tilde q$ over $M$ in the sense of $T^{Sk}$ which is still finitely satisfiable in $A$. Let $b\models \tilde q$ and let $N$ be the Skolem hull of $Mb$. Then $N$ is an extension of $M$. If $c$ is a finite tuple in $N$, then $c=\bar f(b,m)$ for some finite $m\in M$ and tuple $\bar f$ of $\emptyset$-definable functions. Then $\tp(c/M)$ is finitely satisfiable in the Skolem hull of $Am$ which has size $\kappa$.
\end{proof}

\begin{lemme}($T$ is NIP.)
Let $p\in S(M)$ and assume that $p$ is not compressible. Then there is some partial type $\pi$ generically stable over $M$, $\pi|M\subseteq p$, and a conservative extension $q\in S(N)$ of $p$ such that $\pi |N \nsubseteq q$. Furthermore, we can find $N$ such that $|N|=|M|+|T|$.
\end{lemme}
\begin{proof}
Let $a\models p$ and take $(M,a)\prec (M_1,a)$ an $|M|^+$-saturated elementary extension. Define $p_1 = \tp(a/M_1)$. Then $p\subseteq p_1$ is conservative. Let $\pi$ be generically stable over $M$ given by Theorem \ref{th_main} for $p$. Since $p$ is not compressible, by Lemma \ref{lem_equivalentcompr}, there is a type $r\in S_y(M_1)$ finitely satisfiable in $M$ such that $p_1(x)\cup r(y)$ does not imply a complete type over $\emptyset$. Since by Theorem \ref{th_main} $p_1(x)\cup (\pi(x) \otimes r(y))|_{M_1}$ implies a complete type over $M$, if $b\models r$, there is an extension of $p_1$ to $M_1b$, say $p_2$, such that $\pi|Mb\nsubseteq p_2$. By Lemma \ref{lem_skolem}, we can find a model $N\succ M_1$ containing $b$ such that for every finite tuple $c\in N$, $\tp(c/M_1)$ is finitely satisfiable in some set of size $|M|$. Let $q$ be an extension of $p_1$ to $N$ such that $\pi|N\nsubseteq q$. Then $q$ is a conservative extension of $p_1$: Let $\phi(x;y)$ be any formula such that $q_\phi$ is definable. Let $c\in N$ be the parameters used in the definition. Let $A\subseteq M_1$ be a set of size $|M|$ such that $\tp(c/M_1)$ is finitely satisfiable in $A$. Then $(p_1)_\phi$ is $A$-invariant. Let $S_1 \subseteq S_y(A)$ be the set of types $s$ such that $p_1 \vdash \phi(x;d)$ for $d\in s(M_1)$, and let $t$ be in the closure of $S_1$, then as $(M_1,a)$ is $|M|^+$-saturated, there is $d\in t(M_1)$ with $p_1 \vdash \phi(x;d)$. Hence $S_1$ is closed and so is its complement by the same argument. Therefore $(p_1)_\phi$ is definable.

To show the furthermore part, take an elementary substructure $(N',a)\prec (N,a)$ of size $|M|$, where $a\models q$ and $N'$ contains $Mb$.\end{proof}

Recall that a type $p(x)$ is \emph{stable} if there is no formula $\phi(x;y)$, realizations $(a_i)_{i<\omega}$ of $p$ and tuples $(b_i)_{i<\omega}$ such that $\phi(a_i;b_j)$ holds if and only if $i\leq j$. A type is stable if only if all of its extensions are definable. A theory is stable if and only if all types are stable. (See for example \cite{HasOn}.)

\begin{thm}($T$ is NIP.)
Let $p\in S(M)$ be non-stable, then there is an extension $q$ of $p$ which is non-realized and compressible.
\end{thm}
\begin{proof}
As $p$ is not stable, it has an extension which is not definable, so we may assume that $p$ is not definable. We build a chain of conservative extensions $p=p_0 \subseteq p_1 \subseteq \ldots$ of $p$ such that each $p_i$ is over a model $M_i$ of size $|M|+|T|$ and such that the following holds: for each $i$, for every definable partial type $(\phi(x;y),d\phi(y))$ defined over $M_i$ and consistent with $p_i$, either this partial type is not consistent with $p_{i+1}$, or it is consistent with every conservative extension of $p_{i+1}$. This can be done easily: given $p_i$ and $M_i$, list all definable partial types over $M_i$ as $(\pi_j(x):j<\kappa)$. Then build by induction a chain of models $M^ j$ containing $M_i$ and an increasing sequence $p^ j\in S(M^ j)$ of conservative extensions of $p_i$ as follows: set $M^0=M_i$, $p^0 = p_i$. At a successor stage $j+1$, if there is a conservative extension of $p^ j$ which is not consistent with $\pi_j(x)$, let $p^{j+1}\in S(M^{j+1})$ be such an extension, otherwise set $p^{j+1}=p^j$. At limit stages, take the union. At the end, set $M_{i+1}=\bigcup_{j<\kappa} M^j$ and $p_{i+1}=\bigcup_{j<\kappa} p^j$. This has the required properties.

Having done this, let $q = \bigcup_{i<\omega} p_i$. Then $q$ is a conservative extension of $p$. In particular, it is not realized. Let $\pi(x) = (\phi(x,y),d\phi(y))$ be a partial type finitely definable over $M_* := \bigcup M_i$ and consistent with $q$. Then $\pi$ is definable over $M_i$ for some $i$. As $\pi$ is consistent with $p_{i+1}$, by construction, this implies that it is consistent with all conservative extensions of $p_{i+1}$ and a fortiori with all conservative extensions of $q$. Therefore for any Ind-definable partial type $\pi(x)$, defined over $M_*$, if $\pi(x)$ is consistent with $q$, it is consistent with all conservative extensions of $q$. By the previous lemma, this implies that $q$ is compressible.
\end{proof}

\bigskip
\footnotesize
\noindent\textit{Acknowledgments.}
I would like to thank the referee for an extremely comprehensive and helpful report.

This research was partly supported by ValCoMo (ANR-13-BS01-0006) and by NSF (DMS-1665491).

\end{document}